\documentclass[11pt]{amsart}
\usepackage[T1]{fontenc}
\usepackage[utf8]{inputenc}
\usepackage{amsmath,amssymb,amsthm}
\usepackage[margin=1.1in]{geometry}
\usepackage{mathtools}
\usepackage{amsaddr}

\theoremstyle{definition}
\newtheorem{theorem}{Theorem}[section]
\newtheorem{lemma}[theorem]{Lemma}
\newtheorem{proposition}[theorem]{Proposition}

\newtheorem{remark}[theorem]{Remark}

\newtheorem{sublemma}[theorem]{Sublemma}

\newcommand{\W}[1]{W^{#1,1}}
\newcommand{\Lop}[1]{\mathcal{L}_{T_{\varepsilon,#1}}}
\newcommand{\Lz}{\mathcal{L}_{T_{0}}}
\newcommand{\Le}{\mathcal{L}_{\varepsilon}}
\newcommand{\norm}[1]{\left\| #1 \right\|}
\newcommand{\bvt}{\operatorname{BV}_{1,1/2}}
\newcommand{\bv}{\operatorname{BV}_{1,1/p}}
\newcommand{\LT}{\mathcal{L}_{T}}
\newcommand{\abs}[1]{\left| #1 \right|}
\newcommand{\Astar}{A_\ast}

\DeclareMathOperator{\osc}{osc}
\def\var{\operatorname{var}}
\let\eps=\varepsilon

\title[Self-consistent transfer operators for interval maps with singularities]{On the dynamics of self-consistent transfer operators for interval maps with singularities}
\author{Shadiyar Y. Altynbekov$^1$, Umrbek Olimov$^2$, Marks Ruziboev$^{3,4}$}

\address{$^1$South Kazakhstan State University, Tauke Khan Ave 5, Shymkent, Kazakhstan}
\email{altynbekov\_shadiar@mail.ru}
\address{$^{2}$Tashkent University Information Technologies, Amir Temur Avenue 108, Tashkent, Uzbekistan}
\email{umrbek.olimov.92@mail.ru}
\address{$^3$V.I. Romanovskiy Institute of Mathematics, Uzbekistan Academy of Sciences, 9, Universitet street, 100174, Tashkent, Uzbekistan}
\address{$^{4}$National Pedagogical University of Uzbekistan, 27 Bunyodkor street, Tashkent, Uzbekistan}

\email{m.ruziboev@npuuz.uz}
\date{\today}

\begin{document}
\begin{abstract}
    We study the dynamics of self-consistent transfer operators associated with
    mean-field coupled systems whose unit dynamics is given by an interval map with
    finitely many singularities of cusp type. We show that, when the coupling
    strength is sufficiently weak, such an operator has a unique fixed point in a
    suitable cone of the associated Banach space, and that this fixed point attracts
    the orbit of every density in the cone at an exponential rate in the
    corresponding norm.
\end{abstract}
\maketitle

\section{Introduction}Mean-field coupled dynamical systems have been extensively studied in the physics literature as models of collective behaviour in large ensembles of interacting units. In his seminal work Kaneko \cite{Kaneko1990, Kaneko} studied mean-field coupled systems and observed that, even for small coupling strength, the collective dynamics can differ from that of the uncoupled system, a phenomenon he described as a violation of the law of large numbers. Later Chawanya and Morita \cite{ChawanyaMorita1993} studied the mean-field coupling of logistic maps, demonstrating the emergence of synchronised states and non-trivial collective dynamics. The thermodynamic limit of such systems, in which the number of units tends to infinity, leads to a nonlinear evolution equation for the density of states, which may be regarded as a discrete-time analogue of a Vlasov-type equation. In the deterministic setting studied here, this evolution is governed by the self-consistent transfer operator $\mathcal{L}_\varepsilon(h)=\mathcal{L}_{\varepsilon,h}h$, whose fixed points correspond to stationary states of the mean-field dynamics \cite{Keller2006,BardetKellerZweimuller2009}.

The rigorous study of mean-field coupled maps began with the work of Keller \cite{Keller2006}, who proved the existence of invariant densities for the self-consistent transfer operator under suitable conditions. In subsequent work, Bardet, Keller and Zweim\"uller \cite{BardetKellerZweimuller2009} established stochastic stability and identified the thermodynamic limit for globally coupled maps with bistable thermodynamic behaviour. Coupled systems or self-consistent systems remain an active field of research, we refer to  \cite{BFPh, BK, BPh, BL1, BLS, BKST18, CasGalTan1, CasGalTan2, Gal, T22}  and references therein for recent works in this direction.  More recently, Bahsoun and Korepanov \cite{BK} studied statistical properties of mean-field coupled systems of intermittent maps, while Bahsoun and Liverani \cite{BL1} developed a bifurcation theory for mean-field coupled chaotic maps. They apply their theory to globally coupled expanding maps and to globally coupled Axiom A diffeomorphisms, and show that the system undergoes phase transitions as the coupling strength varies.

The present work extends the self-consistent framework of Keller \cite{Keller2006} and of B\'alint, Keller, S\'elley and T\'oth \cite{BKST18}, developed there for unit maps that are uniformly expanding with bounded derivative, to unit maps possessing finitely many singularities of cusp type, where the derivative is unbounded. On  the other hand, it extends the analysis of maps with a single cusp carried out in \cite{BG} and \cite{6boyev} for a fixed map to families with several cusps and to the coupled, self-consistent setting; to our knowledge this case has not been treated before. Our main contribution is a proof of exponential convergence to equilibrium for the self-consistent dynamics in the $W^{1,1}$-norm, under the assumption that the singularities are mild (exponents $\beta_j\in(-1,-3/4)$). The proof combines properties of the spaces $\bv$ of Keller \cite{Ke85} with a Lasota--Yorke inequality in the Sobolev space $W^{1,1}$.

The rest of the paper is organized as follows. In the next section, we state the problem and formulate the main result. Also, we give an example of family satisfying our assumptions. In Section \ref{sec:proofs} we prove the main result. 

\section{Setup}\label{sec:setup}

We consider a family of interval maps $T_{\varepsilon,h}$ whose dynamics depends
self-consistently on the current probability density $h$.

Fix $k\ge1$ and points
\[
0=c_{0}<c_{1}<\cdots<c_{k}<c_{k+1}=1 ,
\]
and set $I_{i}:=(c_{i},c_{i+1})$ for $i=0,\dots,k$, so that
$[0,1]\setminus\{c_{1},\dots,c_{k}\}=\bigcup_{i=0}^{k}I_{i}$. The
points $c_{1},\dots,c_{k}$ are the \emph{singularities}; $c_{0}=0$ and
$c_{k+1}=1$ are ordinary boundary points. It is convenient to fix, once and for
all, the scale
\begin{equation}\label{eq:r0}
r_{0}:=\tfrac13\min_{0\le j\le k}(c_{j+1}-c_{j}),\qquad
U_{j}:=(c_{j}-r_{0},\,c_{j}+r_{0})\quad(1\le j\le k),\qquad
V:=[0,1]\setminus\bigcup_{j=1}^{k}U_{j},
\end{equation}
so that $U_{1},\dots,U_{k}$ are pairwise disjoint, $0,1\in V$, and
$\abs{x-c_{j}}\ge r_{0}$ for every $x\in V$ and every $j$.

Let $m$ denote Lebesgue measure on $[0,1]$. Define 
$$\mathcal B_{1}:=\{h\in L^{1}[0,1]:h\ge0,\ \int h\,dm=1\}$$ 
and fix $\varepsilon_{0}>0$. For
$\eps\in[0,\varepsilon_{0})$ and $h\in\mathcal B_{1}$ let $T_{\eps,h}:[0,1]\to[0,1]$ be
non-singular with respect to $m$. For a given density $h\in\mathcal B_{1}$ the
evolution of the system is determined by the map $T_{\varepsilon,h}$, and the
corresponding density evolves according to the nonlinear transfer operator
\[
\mathcal{L}_{\varepsilon}(h) =\mathcal{L}_{T_{\varepsilon,h}}h,
\]
where $\mathcal{L}_{T_{\varepsilon,h}}$ is the transfer operator associated with
$T_{\varepsilon,h}$, i.e.\ for all $f\in L^{1}$ and $g\in L^{\infty}$,
\[
    \int_{0}^{1} \mathcal{L}_{T_{\varepsilon,h}}f\cdot g\,dm
    =\int_{0}^{1} f\cdot g\circ T_{\varepsilon,h}\,dm.
\]
Throughout the paper we fix $p=2$; the exponent window
$\beta_{j}\in\bigl(-1,-\tfrac34\bigr)$ appearing below is precisely the condition
$\beta_{j}<\tfrac{1-2p}{2p}$ for this value of $p$.

We assume that the family $T_{\eps,h}$ satisfies the conditions (A1)--(A8) below,
with exponents $\beta_{1},\dots,\beta_{k}$. In each of (A4) and (A6)--(A8) the
constants are quantified \emph{before} $\eps$ and $h$; this is what is meant by
saying that they are uniform in $\eps\in[0,\varepsilon_{0})$ and $h\in\mathcal B_{1}$.

\begin{itemize}
\item[(A1)] For every $i\in\{0,\dots,k\}$,
$T_{i,\eps,h}:=T_{\eps,h}|_{I_{i}}$ is one-to-one.

\item[(A2)] $T_{\eps,h}(0)=T_{\eps,h}(1)=0$ and for every
$j=1,\dots,k$ the one-sided limits
$\lim_{x\to c_{j}^{\pm}}T_{\eps,h}(x)=a_{j,\eps,h}\in[0,1]$ exist and
agree, i.e.\ $T_{\eps,h}$ extends continuously across each singular
point.

\item[(A3)]
$T_{\eps,h}|_{[0,1]\setminus\{c_{1},\dots,c_{k}\}}\in C^{3}$.

\item[(A4)] \label{teta}
There exists $\theta>1$ such that for every $\eps\in[0,\varepsilon_{0})$ and every
$h\in\mathcal B_{1}$,
\[
\inf_{x\notin\{c_{1},\dots,c_{k}\}}\abs{T_{\eps,h}'(x)}\ \ge\ \theta .
\]

\item[(A5)] $T_{0}$ is independent of $h$ and is topologically
mixing on $E_{0}:=\bigcup_{i=0}^{k}T_{i,0}(I_{i})$.

\item[(A6)] There exist exponents
$\beta_{j}\in\bigl(-1,-\tfrac34\bigr)$, $j=1,\dots,k$, such that for every
$\eps\in[0,\varepsilon_{0})$, every $h\in\mathcal B_{1}$ and every $j$,
\[
\lim_{x\to c_{j}^{\pm}}T_{\eps,h}'(x)=\pm\infty
\qquad\text{and}\qquad
\lim_{x\to c_{j}^{\pm}}
\frac{\abs{T_{\eps,h}'(x)}}{\abs{x-c_{j}}^{\beta_{j}}}=C^{(j)}_{\eps,h,1}>0.
\]

\item[(A7)] For every $\eps\in[0,\varepsilon_{0})$, every $h\in\mathcal B_{1}$ and every
$j=1,\dots,k$,
\[
\lim_{x\to c_{j}^{\pm}}
\frac{\abs{T_{\eps,h}''(x)}}{\abs{x-c_{j}}^{\beta_{j}-1}}=C^{(j)}_{\eps,h,2}>0,
\qquad
\lim_{x\to c_{j}^{\pm}}
\frac{\abs{T_{\eps,h}'''(x)}}{\abs{x-c_{j}}^{\beta_{j}-2}}=C^{(j)}_{\eps,h,3}>0.
\]

\item[(A8)] \label{unif}
There exist constants $0<C_{-}\le C_{+}<\infty$ such that for
every $\eps\in[0,\varepsilon_{0})$ and every $h\in\mathcal B_{1}$:
\begin{itemize}
\item[(i)] for every $j\in\{1,\dots,k\}$, every $i\in\{0,1,2\}$ and every
$x\in U_{j}\setminus\{c_{j}\}$,
\[
\bigl|T^{(i+1)}_{\eps,h}(x)\bigr|\le C_{+}\abs{x-c_{j}}^{\beta_{j}-i}
\qquad\text{and}\qquad
\abs{T_{\eps,h}'(x)}\ge C_{-}\abs{x-c_{j}}^{\beta_{j}} ;
\]
\item[(ii)] for every $i\in\{0,1,2\}$ and every $x\in V$,
$\ \bigl|T^{(i+1)}_{\eps,h}(x)\bigr|\le C_{+}$.
\end{itemize}
\item[(A9)]\label{a:A9p}
There exists $\hat C>0$ such that for every $\varepsilon\in[0,\varepsilon_{0})$,
every admissible density $h$ and every $g\in\W{1}$,
\[
\norm{(\Lop{h}-\Lz)\,g}_{2}
\le \hat C\,\varepsilon\,\bigl(1+\norm{h}_{2}\bigr)\norm{g}_{\W{1}} .
\]
\item[(A10)]\label{a:A10}
There exists $\hat C>0$ such that for every $\varepsilon\in[0,\varepsilon_{0})$,
all admissible densities $h_{1},h_{2}\in\W{1}$ and every $g\in\W{2}$,
\[
\norm{(\Lop{h_{1}}-\Lop{h_{2}})\,g}_{\W{1}}
\le \hat C\,\varepsilon\,\norm{h_{1}-h_{2}}_{\W{1}}\,\norm{g}_{\W{2}} .
\]
\end{itemize}
Notice that this family is a direct generalization of the maps studied in \cite{BG};
see also \cite{6boyev} for a similar family. The exact values of the limiting constants
$C^{(j)}_{\eps,h,i}$ in (A6)--(A7) are not used, we only use the uniform bounds of (A8). Further, from the statement it is clear that  $\hat C$ does not depend on $\eps$, on the densities or on $g$. An example of families satisfying (A9)--(A10) is given in
Subsection~\ref{ss:examples}.

\medskip 
The following is the main result of the current work. 

\begin{theorem}\label{thm:main}
There exist constants $\varepsilon_{*}>0$, $K_{1},K'\ge1$, $C>0$, $\rho\in(0,1)$, and
\[
\mathcal D:=\Bigl\{h\ge0:\ \textstyle\int h\,dm=1,\
\norm h_{\W1}\le K_{1}\Bigr\}\subset \W1,
\]
such that for every $\varepsilon\in[0,\varepsilon_{*})$ the coupled system admits a
density $h_{\varepsilon}\in\W{1}$ with $\norm{h_{\varepsilon}}_{\W{1}}\le K_{1}$
and $\norm{h_{\varepsilon}}_{1,1/2}\le K'$ satisfying the following properties.
\begin{itemize}
\item[(i)] $\Le(h_{\varepsilon})=h_{\varepsilon}$, which is the unique fixed point of $\Le$ in $\mathcal D$;
\item[(ii)] for every $h\in\mathcal D$,
\[
\norm{\Le^{n}(h)-h_{\varepsilon}}_{\W{1}}\le C\rho^{\,n},
\qquad n\in\mathbb N .
\]
\end{itemize}
\end{theorem}

\begin{remark}\label{rem:DvsB}
 We don't claim $\Le$  to be a contraction on $\mathcal D$, which is appearing as the \emph{basin} in (i) and (ii).
The
contraction estimate i.e., Lemma~\ref{lem:stab}, is asymmetric: its first argument is required to lie in the smaller set
$\bar B\subset \W2$ defined in equation \eqref{eq:barB}, while
its second argument must only lie in $\mathcal D$. 
Therefore, in the above theorem, by working on a smaller set $\bar B$ we first estimate $\varepsilon_*$. Once we estimate $\varepsilon_*$ and fix it, we obtain the convergence result in Theorem \ref{thm:main}.
Further, we  have
$\bar B\subset\mathcal D$, and we do not know whether $h_{\varepsilon}$ belongs
to $\W2$.
\end{remark}

\subsection{A class of examples}\label{ss:examples}

Assumptions (A9)--(A10) are perturbative statements about the dependence of
$\Lop h$ on the coupling, and it is natural to ask whether they hold for concrete self-consistent families. We describe such a class below. 

Let $u_{0}>0$ and let $S:[0,1]\times(-u_{0},u_{0})\to[0,1]$ be a family of maps
such that
\begin{itemize}
\item[(E1)] $S(\cdot,0)=T_{0}$, and for every $u\in(-u_{0},u_{0})$ the map
$S(\cdot,u)$ satisfies (A1)--(A4) and (A6)--(A8) with the \emph{same} singular
points $c_{j}$, the \emph{same} exponents $\beta_{j}$, and with constants
$\theta,C_{\pm}$ independent of $u$; $T_0$ satisfies (A5).
\item[(E2)] the assignment $u\mapsto\mathcal L_{S(\cdot,u)}$ is Lipschitz from
$(-u_{0},u_{0})$ into the bounded operators $\W1\to L^{2}$, with constant $L_{1}$,
and into the bounded operators $\W2\to\W1$, with constant $L_{2}$.
\end{itemize}
Fix a kernel $\kappa\in C^{3}([0,1])$ with $\norm\kappa_{\infty}\le1$ and couple
the units through the scalar mean field
\[
\Phi(h):=\int_{0}^{1}\kappa(y)h(y)\,dm(y),\qquad h\in\mathcal B_{1},
\]
so that $\abs{\Phi(h)}\le1$ for every $h\in\mathcal B_{1}$. For
$\eps\in[0,\varepsilon_{0})$ with $\varepsilon_{0}<u_{0}$ put
\begin{equation}\label{eq:example}
T_{\eps,h}:=S\bigl(\cdot\,,\ \eps\,\Phi(h)\bigr).
\end{equation}

\begin{proposition}\label{prop:examples}
Assume (E1)--(E2) and that $T_{0}$ is topologically mixing on $E_{0}$. Then the
family \eqref{eq:example} satisfies (A1)--(A10),  where $\hat C$ in (A9)--(A10) will be $\max\{L_{1},L_{2}\}$.
\end{proposition}

\begin{proof}
For each fixed $\eps$ and $h$ the map $T_{\eps,h}$ is the member $S(\cdot,u)$ of
the family with parameter $u=\eps\Phi(h)\in(-u_{0},u_{0})$, so (A1)--(A3) and
(A6)--(A8) hold with constants independent of $(\eps,h)$ by (E1). Moreover
$T_{0,h}=S(\cdot,0)=T_{0}$ is independent of $h$, so (A5) holds. Assumption (E1) implies (A4), since the constant $\theta$ is independent
of $u$.

We show (A9) also holds. Let $g\in\W1$. Since $\abs{\eps\Phi(h)}\le\eps$, the first Lipschitz
property in (E2) gives
\[
\norm{(\Lop h-\Lz)g}_{2}
=\norm{\bigl(\mathcal L_{S(\cdot,\eps\Phi(h))}-\mathcal L_{S(\cdot,0)}\bigr)g}_{2}
\le L_{1}\eps\norm g_{\W1}
\le L_{1}\eps\bigl(1+\norm h_{2}\bigr)\norm g_{\W1}.
\]
Finally, we show (A10). Let $h_{1},h_{2}\in\mathcal B_{1}\cap\W1$ and $g\in\W2$. Since $\Phi$
is linear and $\norm\kappa_{\infty}\le1$,
\[
\abs{\eps\Phi(h_{1})-\eps\Phi(h_{2})}
=\eps\abs{\Phi(h_{1}-h_{2})}
\le\eps\norm{h_{1}-h_{2}}_{1}\le\eps\norm{h_{1}-h_{2}}_{\W1},
\]
so the second Lipschitz property in (E2) yields
\[
\norm{(\Lop{h_{1}}-\Lop{h_{2}})g}_{\W1}
\le L_{2}\,\abs{\eps\Phi(h_{1})-\eps\Phi(h_{2})} \norm g_{\W2} \le L_{2}\,\eps\,\norm{h_{1}-h_{2}}_{\W1}\norm g_{\W2}. \qedhere
\]
\end{proof}

\begin{remark}\label{rem:examples}
 One can also give a concrete example of maps satisfying the above conditions. Consider 
\[
S(x,u):=T_{0}(x)+u\,\psi(x),
\]
where $T_0$ is the explicit example given in \cite[Section 4]{BG}, $\psi\in C^{3}([0,1])$ is supported in a compact subset $K$ of the interior
of $V$ with $T_{0}(K)\subset(0,1)$. Then $S(\cdot,u)$ coincides with $T_{0}$ on a
neighbourhood of every singularity and of $\{0,1\}$, so (A2) and the local
conditions (A6)--(A8)(i) are unaffected by $u$. Further, there exists $u_0>0$ such that $\abs{S'(\cdot,u)}\ge\theta- u_0\norm{\psi'}_{\infty}>1$ for all $\abs u\le u_0$ and
$S(\cdot,u)$ maps $[0,1]$ into itself, giving (A1), (A4) and (A8)(ii). In this
case (E2) reduces to a standard smooth-perturbation estimate for the transfer
operator, since the perturbation is supported away from the singular set.
\end{remark}

\section{Proofs}\label{sec:proofs}
The following proposition is the main technical result, which states the key estimates we need. Its proof is somewhat standard and is therefore postponed to the appendix.

\begin{proposition}\label{prop:LY}
Assume that the family $T_{\eps,h}:[0,1]\to[0,1]$ satisfies
{\rm(A1)--(A8)} for all $\eps\in[0,\varepsilon_{0})$ and
$h\in\mathcal B_{1}$. Then there are $\lambda\in(0,1)$ and $M_{0}\ge0$,
depending only on $k$, $\theta$, the points $c_{j}$, the exponents
$\beta_{j}$ and the constants in {\rm(A8)} 
%--- in particular uniform in $\eps$ and $h$ --- 
such that for every $\eps\in[0,\varepsilon_{0})$
and $h\in\mathcal B_{1}$:
\begin{align}
&\norm{(\Lop hf)'}_{1}\le\lambda\norm{f'}_{1}+M_{0}\norm f_{2}
&&\text{for all }f\in\W1,
\label{eq:LY1}\\
&\norm{\Lop hf}_{\W2}\le\lambda^{2}\norm f_{\W2}+M_{0}\norm f_{\W1}
&&\text{for all }f\in\W2,
\label{eq:LY2}\\
&\norm{\Lop h}_{L^{2}\to L^{2}}\le
A = \Bigl(\tfrac{k+1}{\theta}\Bigr)^{1/2}.
\label{eq:L2}
\end{align}
Notice that one may take $\lambda=\theta^{-1}$. Also, from the statement it is clear that the constants appearing in \eqref{eq:LY1}-\eqref{eq:L2} are independent of $\eps$ and $h$, but they may depend on $\varepsilon_{0}>0$.  
\end{proposition}

Throughout we set
\begin{equation}\label{eq:Astar}
A_\ast:=\max\{A,1\}=\max\Bigl\{\bigl(\tfrac{k+1}{\theta}\bigr)^{1/2},
\,1\Bigr\}\ \ge1 ,
\end{equation}
so that, by submultiplicativity, any composition of $n$ transfer
operators of the family has $L^{2}\to L^{2}$ norm at most
$A_\ast^{\,n}$.

Throughout the proof we fix $\eps$ and $h$, and track that every constant produced depends only on
the uniform constant listed in the Proposition \ref{prop:LY}.

\begin{remark}\label{rem:window}
The proof uses only $\beta_{j}<-\tfrac12$ for every $j$, which is
implied by the assumed window $\beta_{j}\in(-1,-\tfrac34)$. (The full strength $\beta_{j}<-\tfrac34$ is needed in the Appendix.)  Define
\[
\beta_{*}:=\max_{1\le j\le k}\beta_{j}\in\Bigl(-1,-\tfrac12\Bigr).
\]
\end{remark}

%We also use the results of the Appendix: the space
%$\bv\subset L^{1}$ of Keller \cite{Ke85}, the uniform Lasota--Yorke
%inequality on $\bv$ (Proposition~\ref{prop:LYBVapp}),
%\begin{equation}
%\norm{\Lop{h}f}_{1,1/p}\le\vartheta\norm{f}_{1,1/p}+C_{\mathrm{ap}}\norm{f}_{1},
%\qquad \vartheta\in(0,1),
%\tag{LY-BV}\label{eq:LYBV}
%\end{equation}
%the embedding $\norm{\cdot}_{\infty}\le
%\hat C_{\infty}\norm{\cdot}_{1,1/p}$, and the fact that closed balls of $\bv$
%are closed in $L^{1}$ (Lemma~\ref{lem:lsc}).

\subsection{Auxiliary estimates}
In this subsection we obtain several auxiliary estimates for the iterates of the
transfer operators. 
\begin{lemma}\label{lem:comp1}
Let $M:=M_{0}+(1-\lambda)$ and $M_{n}:=(A_\ast+1)^{n}M$. For all
$n\in\mathbb N$, all $\eps\in[0,\varepsilon_{0})$, all admissible densities
$h_{1},\dots,h_{n}$ and all $g\in\W1$,
\[
\norm{\Lop{h_{n}}\cdots\Lop{h_{1}}g}_{\W1}
\le\lambda^{n}\norm g_{\W1}+M_{n}\norm g_{2}.
\]
\end{lemma}

\begin{proof}
Recall the standard estimates
\begin{equation}\label{eq:embed}
\norm{f}_{1}\le\norm{f}_{2}\le\norm{f}_{\infty}\le
\norm{f}_{\W{1}},\qquad f\in\W{1}([0,1]),
\end{equation}

By \eqref{eq:LY1} and \eqref{eq:embed}, we have 
\begin{equation}\label{eq:s1}
\begin{aligned}
\norm{\Lop{h}f}_{\W{1}}
&\le\lambda\norm{f'}_{1}+M_{0}\norm{f}_{2}+\norm{f}_{1}\\
&=\lambda\norm{f}_{\W{1}}+M_{0}\norm{f}_{2}+(1-\lambda)\norm{f}_{1}
\le\lambda\norm{f}_{\W{1}}+M\norm{f}_{2},
\end{aligned}
\end{equation}
with $M = M_0+(1-\lambda)$.
We now argue by induction on $n$. The case $n=1$ holds since
$M\le(A_\ast+1)M=M_{1}$. Assume the claim for $n-1$. The inequality \eqref{eq:L2} implies that 
$$\norm{\Lop{h_{n-1}}\cdots\Lop{h_{1}}g}_{2}\le
\Astar^{\,n-1}\norm g_{2},$$
Combining the latter with  \eqref{eq:s1}, the inductive step we have 
\[
\norm{\Lop{h_{n}}\cdots\Lop{h_{1}}g}_{\W1}
\le\lambda^{n}\norm g_{\W1}
+\bigl(\lambda M_{n-1}+M\Astar^{\,n-1}\bigr)\norm g_{2}.
\]
Finally, using $\lambda<1$, $\Astar\ge1$ we obtain 
\[
\lambda M_{n-1}+M\Astar^{\,n-1}
\le(\Astar+1)^{n-1}M+\Astar(\Astar+1)^{n-1}M
=(\Astar+1)^{n}M=M_{n}. \qedhere
\]
which finishes the proof.
\end{proof}

\begin{remark}
The constant $M_{n}$ is \emph{exponential} in $n$.  This is harmless in what follows because we fix $n$  always before the parameters that beat  $M_{n}$ are chosen.
\end{remark}
Now we obtain uniform bounds on the iterates of the operator $\Lop{h}$.
\begin{lemma}\label{lem:unif1}
There exists $C_{u}\ge1$ such that for all $n$, $\varepsilon$, admissible
densities $h_{1},\dots,h_{n}\in \W{1}$ and $f\in\W{1}$ the following holds
\[
\norm{\Lop{h_{n}}\cdots\Lop{h_{1}}f}_{\W{1}}\le C_{u}\norm{f}_{\W{1}} .
\]
\end{lemma}

\begin{proof}
Write $F_{\ell}:=\Lop{h_{\ell}}\cdots\Lop{h_{1}}f$ for $1\le\ell\le n$, with
$F_{0}:=f$. By \eqref{eq:s1} we have 
$$\norm{F_{\ell}}_{\W{1}}\le\lambda\norm{F_{\ell-1}}_{\W{1}}
+M\norm{F_{\ell-1}}_{2}.$$ Since $\W1\subset\bvt$, Proposition~\ref{prop:supbd}
applies to every $f\in\W{1}$ and gives a constant $\hat M\ge1$, uniform in $n$, in
$\varepsilon$ and in the admissible densities $h_{1},\dots,h_{n}$, such that
\begin{equation}
\norm{F_{\ell}}_{\infty}\le\hat M\norm{f}_{\infty}, \qquad 0\le\ell\le n.
\label{eq:Uinf}
\end{equation}
Combining \eqref{eq:embed} with \eqref{eq:Uinf} gives
$\norm{F_{\ell-1}}_{2}\le\norm{F_{\ell-1}}_{\infty}\le\hat M\norm{f}_{\infty}
\le\hat M\norm{f}_{\W{1}}$. Iterating,
\[
\norm{F_{n}}_{\W{1}}
\le\lambda^{n}\norm{f}_{\W{1}}
+M\hat M\norm{f}_{\W{1}}\sum_{\ell=0}^{n-1}\lambda^{\ell}
\le\Bigl(1+\frac{M\hat M}{1-\lambda}\Bigr)\norm{f}_{\W{1}},
\]
so the lemma holds with $C_{u}:=1+\dfrac{M\hat M}{1-\lambda}$.
\end{proof}

\begin{lemma}\label{lem:comp2}
There is $C_{2}\ge0$ such that for all $n$, $\varepsilon$, admissible
densities $h_{1},\dots,h_{n}$ and $g\in\W{2}$,
\[
\norm{\Lop{h_{n}}\cdots\Lop{h_{1}}g}_{\W{2}}
\le\lambda^{2n}\norm{g}_{\W{2}}+C_{2}\norm{g}_{\W{1}} .
\]
\end{lemma}

\begin{proof}
Let $b_{k}:=\norm{\Lop{h_{k}}\cdots\Lop{h_{1}}g}_{\W{2}}$. Applying first 
\eqref{eq:LY2} and then Lemma~\ref{lem:unif1} we obtain
\[
b_{k}\le\lambda^{2}b_{k-1}+M_{0}\norm{\Lop{h_{k-1}}\cdots\Lop{h_{1}}g}_{\W{1}}
\le\lambda^{2}b_{k-1}+M_{0}C_{u}\norm{g}_{\W{1}} .
\]
Iterating gives
$$b_{n}\le\lambda^{2n}\norm{g}_{\W{2}}
+\frac{M_{0}C_{u}}{1-\lambda^{2}}\norm{g}_{\W{1}},$$ with
$C_{2}:=M_{0}C_{u}/(1-\lambda^{2})$.
\end{proof}

We are now ready to prove spectral gap for the unperturbed operator. This is one of the key results for obtaining memory loss results below. 

\begin{proposition}\label{prop:gap}
There exist $C_{\mathrm{gap}}\ge1$ and $\sigma\in(0,1)$ such that for every
$g\in\W{1}$ with $\int g\,dm=0$,
\[
\norm{\Lz^{\,m}g}_{\W{1}}\le C_{\mathrm{gap}}\sigma^{m}\norm{g}_{\W{1}},
\qquad m\in\mathbb N .
\]
\end{proposition}

\begin{proof}
By the one-step inequality \eqref{eq:s1} the operator
$\Lz$ satisfies $$\norm{\Lz f}_{\W1}\le\lambda\norm f_{\W1}+M\norm f_{2}$$ and by \eqref{eq:L2} it is bounded on $L^{2}$. The
embedding $\W{1}(0,1)\hookrightarrow L^{2}(0,1)$ is compact by the Rellich–Kondrakov theorem. Therefore, Lemma 2.2 of  \cite{BarGK} implies quasi-compactness of
$\Lz$ on $\W{1}$, with essential spectral radius at most $\lambda<1$. Thus, any point of the spectrum outside the ball $|z|>\lambda$ must be an eigenvalue. Since $1\circ T_0=1$, 1 is in the spectrum. This implies that $(L_{T0})^*1=1$ and hence an eigenvalue. 
Since $\Lz$ is a power-bounded, quasi-compact operator on $W^{1,1}$ preserving the integral, its spectral radius is $r(\Lz)=1$ (We refer to \cite{AB, Baladi} for general properties of transfer operators.) 

By Lemma~\ref{lem:UBV} the
inverse derivative $w_{T_{0}}$ lies in $\mathrm{UBV}_{2}$, and
$\norm{w_{T_{0}}}_{\infty}\le\theta^{-1}<1$ by (A4). Therefore \cite[Theorem 3.3 \& 3.5]{Ke85} applies to $T_{0}$:
the peripheral spectrum is finite, and there are finitely many non-negative
densities $\varphi_{1},\dots,\varphi_{r}$ with pairwise disjoint supports
$\Sigma_{1},\dots,\Sigma_{r}$, each equal modulo $m$ to a finite union of
non-degenerate intervals, which $T_{0}$ permutes; the peripheral eigenvalues are
roots of unity whose orders are the lengths of the cycles of this permutation.

Now we show that $1$ is a simple eigenvalue.  Suppose the
decomposition were non-trivial, i.e.\ either $r\ge2$ or a cycle has length
$q\ge2$. Notice that, each
$\varphi_{s}$ is a non-negative element of $\bv$ with $\int\varphi_{s}\,dm>0$, $s=1,...,r$. Therefore, there exists  a non-degenerate interval on which
$\varphi_{s}>0$ almost everywhere. Moreover, every invariant density is supported
in $\overline{E_{0}}$ modulo $m$: by the branch representation
\eqref{eq:explicit}, $\Lz f$ vanishes almost everywhere outside
$E_{0}=\bigcup_{i=0}^{k}T_{i,0}(I_{i})$ for every $f\in L^{1}$, hence so does
$\Lz\varphi_{s}=\varphi_{s}$. Consequently each $\Sigma_{s}$ contains, modulo
$m$, a non-degenerate interval of $E_{0}$.
Then there are two of the sets above, say $\Sigma$ and $\Sigma'$,
disjoint modulo $m$, and $q\ge2$ with
\begin{equation}\label{eq:cyclic}
m\bigl(T_{0}^{-n}(\Sigma')\cap\Sigma\bigr)=0
\qquad\text{for every $n$ outside a fixed residue class modulo }q .
\end{equation}
By the previous paragraph we may choose non-empty open intervals
$O\subseteq\Sigma\cap E_{0}$ and $O'\subseteq\Sigma'\cap E_{0}$, after
discarding null sets. By (A5) there is
$N$ with $T_{0}^{-n}(O')\cap O\ne\emptyset$ for every $n\ge N$. Fix such an
$n\ge N$ lying outside the distinguished residue class and pick
$x\in T_{0}^{-n}(O')\cap O$. The point $x$ belongs to some interval $J$ of
monotonicity of $T_{0}^{n}$, on which $T_{0}^{n}$ is continuous; hence
$(T_{0}^{n}|_{J})^{-1}(O')$ is relatively open in $J$ and contains $x$, so
$T_{0}^{-n}(O')\cap O$ contains a non-empty open interval and therefore has
positive Lebesgue measure. This contradicts \eqref{eq:cyclic}. Hence $r=1$ and
$q=1$: the eigenvalue $1$ is simple and there is no further peripheral spectrum.

It is now standard to  write $\Lz=\Pi+Q$, where
\[
\Pi f:=\Bigl(\int_{0}^{1}f\,dm\Bigr)\varphi_{1},
\qquad\int_{0}^{1}\varphi_{1}\,dm=1,
\]
is the rank-one spectral projection associated with the eigenvalue $1$ and  $Q:=\Lz(\mathrm{Id}-\Pi)$, so that $\Pi Q=Q\Pi=0$. By the previous step  the spectral
radius of $Q$ on $\W1$ is some $\sigma_{0}<1$. Fix $\sigma\in(\sigma_{0},1)$. The
spectral radius formula gives $C_{\mathrm{gap}}\ge1$ with
$\norm{Q^{m}}_{\W1\to\W1}\le C_{\mathrm{gap}}\sigma^{m}$ for every $m$. If
$g\in\W1$ has $\int g\,dm=0$ then $\Pi g=0$, hence $\Lz^{\,m}g=Q^{m}g$ and
$$\norm{\Lz^{\,m}g}_{\W1}\le C_{\mathrm{gap}}\sigma^{m}\norm g_{\W1}.$$
\end{proof}

\subsection{Memory loss} Next we show exponential memory loss for  zero average functions. For $K\ge0$ let
\[
B(K):=\Bigl\{g\in L^{1}: g\ge0,\ \int g\,dm=1,\
\norm{g}_{\W{1}}\le K\Bigr\}.
\]
\begin{lemma}\label{lem:memloss}
Fix $K\ge0$. There exist $\varepsilon_{0}=\varepsilon_{0}(K)>0$,
$C_{\mathrm{ml}}=C_{\mathrm{ml}}(K)\ge1$ and $\gamma=\gamma(K)\in(0,1)$
such that for every $\varepsilon\in[0,\varepsilon_{0})$, every
$n\in\mathbb N_{0}$, all densities $h_{1},\dots,h_{n}\in B(K)$ and
every $g\in\W{1}$ with $\int g\,dm=0$,
\[
\norm{\Lop{h_{n}}\cdots\Lop{h_{1}}g}_{\W{1}}
\le C_{\mathrm{ml}}\,\gamma^{\,n}\norm{g}_{\W{1}} .
\]
\end{lemma}

\begin{proof}
Since $h_{i}\in B(K)$, by \eqref{eq:embed}
\begin{equation}\label{eq:hb}
\norm{h_{i}}_{2}\le K .
\end{equation}

Let $n,k\in\mathbb N$ and let $g$ be
mean-zero. Splitting off the last $n$ operators,
Lemma~\ref{lem:comp1} gives
\begin{equation}\label{eq:split}
\norm{\Lop{h_{n+k}}\cdots\Lop{h_{1}}g}_{\W{1}}
\le\lambda^{n}\norm{\Lop{h_{k}}\cdots\Lop{h_{1}}g}_{\W{1}}
+M_{n}\norm{\Lop{h_{k}}\cdots\Lop{h_{1}}g}_{2},
\end{equation}
and the first term is $\le\lambda^{n}C_{u}\norm{g}_{\W{1}}$ by
Lemma~\ref{lem:unif1}. For the second term, the telescoping identity
\[
\Lop{h_{k}}\cdots\Lop{h_{1}}-\Lz^{\,k}
=\sum_{i=1}^{k}\Lop{h_{k}}\cdots\Lop{h_{i+1}}
\bigl(\Lop{h_{i}}-\Lz\bigr)\Lz^{\,i-1}
\]
together with \eqref{eq:L2} yields
\[
\norm{\Lop{h_{k}}\cdots\Lop{h_{1}}g}_{2}
\le\sum_{i=1}^{k}\Astar^{\,k-i}
\norm{(\Lop{h_{i}}-\Lz)\Lz^{\,i-1}g}_{2}+\norm{\Lz^{\,k}g}_{2}.
\]
By (A9), \eqref{eq:hb} and Lemma~\ref{lem:unif1},
\[
\norm{(\Lop{h_{i}}-\Lz)\Lz^{\,i-1}g}_{2}
\le\hat C\varepsilon(1+K)\norm{\Lz^{\,i-1}g}_{\W{1}}
\le\hat C\varepsilon(1+K)C_{u}\norm{g}_{\W{1}},
\]
while, since $g$ is mean-zero, Proposition~\ref{prop:gap} and
\eqref{eq:embed} give
$\norm{\Lz^{\,k}g}_{2}\le C_{gap}\sigma^{k}\norm{g}_{\W{1}}$.
Substituting into \eqref{eq:split} and using
$\sum_{i=1}^{k}\Astar^{k-i}\le (1+\Astar)^{k}$,
\begin{equation}\label{eq:block}
\norm{\Lop{h_{n+k}}\cdots\Lop{h_{1}}g}_{\W1}
\le\Bigl[\lambda^{n}C_{u}
+M_{n}\hat C(1+K)C_{u}(\Astar+1)^{k}\eps
+M_{n}C_{gap}\sigma^{k}\Bigr]\norm g_{\W1}.
\end{equation}

Set $\gamma_{*}:=\tfrac12$ and choose the constants in the following order.
First fix $n_{*}$ with $\lambda^{n_{*}}C_{u}\le\gamma_{*}/3$; this also fixes
$M_{n_{*}}=(\Astar+1)^{n_{*}}M$. Next fix $k_{*}$ so that
$M_{n_{*}}C_{gap}\sigma^{k_{*}}\le\gamma_{*}/3$. Finally, we choose
$\eps_{0}(K)>0$ so small that
$M_{n_{*}}\hat C(1+K)C_{u}(\Astar+1)^{k_{*}}\eps\le\gamma_{*}/3$ for every
$\eps<\eps_{0}(K)$. Because $n_{*}$ and $k_{*}$ are fixed before $\eps$, the
exponential growth of the constants in $\Astar$ causes no problem. Setting
$N:=n_{*}+k_{*}$, the bound \eqref{eq:block} yields
\begin{equation}\label{eq:oneblock}
\norm{\Lop{h_{N}}\cdots\Lop{h_{1}}g}_{\W1}\le\gamma_{*}\norm g_{\W1}
\end{equation}
for every $\varepsilon\in[0,\varepsilon_{0})$, every mean-zero $g\in\W{1}$ and
all $h_{1},\dots,h_{N}\in B(K)$.

It remains to iterate this estimate. 
Let $n=qN+r$, $0\le r<N$. Since each block acts on a mean-zero function, the upper bound 
\eqref{eq:oneblock} applies $q$ times; the remaining $r$ operators are
absorbed by Lemma~\ref{lem:unif1}:
\[
\norm{\Lop{h_{n}}\cdots\Lop{h_{1}}g}_{\W{1}}
\le C_{u}\gamma_{*}^{\,q}\norm{g}_{\W{1}}
\le\frac{C_{u}}{\gamma_{*}}\bigl(\gamma_{*}^{1/N}\bigr)^{n}
\norm{g}_{\W{1}},
\]
since $q\ge n/N-1$. Setting $\gamma:=\gamma_{*}^{1/N}\in(0,1)$ and
$C_{\mathrm{ml}}:=\max(1,C_{u}/\gamma_{*})$ finishes the proof.
\end{proof}

\subsection{An exponentially stable invariant set}
We now construct an invariant set in $W^{2,1}$ with respect to the transfer operators. Set 
\begin{equation}\label{eq:constants}
K':=\max\Bigl(1,\frac{C_{\mathrm{ap}}}{1-\vartheta}\Bigr),\qquad
K_{1}:=\max\Bigl(1,\frac{M_{0}\hat C_{\infty}K'+1}{1-\lambda}\Bigr),\qquad
K_{2}:=\max\Bigl(1,\frac{M_{0}K_{1}}{1-\lambda^{2}}\Bigr),
\end{equation}
and define
\begin{equation}\label{eq:barB}
\bar B:=\Bigl\{g\ge0:\ \int g\,dm=1,\
\norm{g}_{1,1/p}\le K',\ \norm{g}_{\W{1}}\le K_{1},\
\norm{g}_{\W{2}}\le K_{2}\Bigr\}.
\end{equation}

\begin{proposition}\label{prop:inv}
$\bar B$ is nonempty and convex, $\mathbf 1\in\bar B$, and
$\Le(\bar B)\subseteq\bar B$ for every $\varepsilon\in[0,\varepsilon_{0})$.
Moreover, for every density $h$ with $\norm{h}_{\W{1}}\le K_{1}$ the
whole trajectory satisfies
$$\sup_{k\ge0}\norm{\Le^{k}(h)}_{\W{1}}\le C_{u}K_{1}=:K_{*}.$$
\end{proposition}

\begin{proof}
$\mathbf 1$ satisfies $\norm{\mathbf 1}_{1,1/p}
=\norm{\mathbf1}_{\W{1}}=\norm{\mathbf1}_{\W{2}}=1$, so $\mathbf1\in\bar B$ by \eqref{eq:constants}. Convexity
is clear since each defining condition is convex. Let $g\in\bar B$.
Recall that $\Le(g)=\Lop{g}g$, which is again a density. By
Proposition~\ref{prop:LYBVapp},
\[
\norm{\Le g}_{1,1/p}\le\vartheta K'+C_{\mathrm{ap}}\le K'
\]
by the choice of $K'$. Next, $\norm{\Le g}_{1}=1$ since $\Le g$ is a density.
Therefore, using
$\norm{g}_{2}\le\norm{g}_{\infty}\le\hat C_{\infty}\norm{g}_{1,1/p}
\le\hat C_{\infty}K'$ implies 
\[
\norm{\Le g}_{\W{1}}
\le\lambda\norm{g'}_{1}+M_{0}\norm{g}_{2}+1
\le\lambda K_{1}+M_{0}\hat C_{\infty}K'+1\le K_{1}
\]
by the choice of $K_{1}$. Finally, by \eqref{eq:LY2},
$\norm{\Le g}_{\W{2}}\le\lambda^{2}K_{2}+M_{0}K_{1}\le K_{2}$ by the
choice of $K_{2}$.  Lemma~\ref{lem:unif1}
 implies that
$\norm{\Le^{k}h}_{\W{1}}\le C_{u}\norm{h}_{\W{1}}$.
\end{proof}

\begin{remark}\label{rem:compact}
We emphasise that we do \emph{not} use compactness of $\bar B$ or of $B(K)$, in
$L^{2}$ or elsewhere. The fixed point is produced below by a Cauchy-sequence
argument, which requires no compactness. The only closedness fact we need is that
the ball $\{\norm{g}_{1,1/p}\le K'\}$ is closed in $L^{1}$, and we deduce this in
Lemma~\ref{lem:lsc} from the lower semi-continuity of
$\operatorname{var}_{1,1/p}$ along $L^{1}$-convergent sequences.
\end{remark}

We next prove some technical lemma, which is a discrete analogue of Gr\"onwal's lemma, that will be useful below in the iteration process.  

\begin{lemma}\label{lem:gronwall}
Let a sequence of real numbers  $(d_{n})_{n\ge0}\subset[0,\infty)$, and numbers $a\ge1$, $\gamma\in(0,1)$ and  $b\ge0$ be given. Assume that for all $n\ge 1$ the following inequality holds
$$d_{n}\le a\gamma^{n}d_{0}
+b\sum_{i=1}^{n}\gamma^{\,n-i}d_{i-1}.$$
If $b\le\frac{1-\gamma}{4}$ then 
$d_{n}\le2a\rho^{\,n}d_{0}$ holds for any $n\ge 1$ with $\rho:=\frac{1+\gamma}{2}$.
\end{lemma}

\begin{proof}
We argue by induction on $n$. For $n=0$ the claim reads $d_{0}\le2ad_{0}$, which
holds because $a\ge1$. Let $n\ge1$ and assume $d_{i}\le2a\rho^{\,i}d_{0}$ for all
$i<n$. Since $\rho-\gamma=\frac{1-\gamma}{2}$ and $\gamma<\rho$,
\[
\sum_{i=1}^{n}\gamma^{\,n-i}\rho^{\,i-1}
=\rho^{\,n-1}\sum_{i=1}^{n}\Bigl(\frac{\gamma}{\rho}\Bigr)^{n-i}
\le\frac{\rho^{\,n-1}}{1-\gamma/\rho}
=\frac{\rho^{\,n}}{\rho-\gamma}
=\frac{2\rho^{\,n}}{1-\gamma}.
\]
Substituting the inductive hypothesis into the assumed inequality and using
$b\le\frac{1-\gamma}{4}$,
\[
d_{n}\le a\gamma^{\,n}d_{0}+2ab\,d_{0}\sum_{i=1}^{n}\gamma^{\,n-i}\rho^{\,i-1}
\le a\gamma^{\,n}d_{0}+\frac{4ab}{1-\gamma}\rho^{\,n}d_{0}
\le a\rho^{\,n}d_{0}+a\rho^{\,n}d_{0}=2a\rho^{\,n}d_{0},
\]
where we used $\gamma<\rho$ in the last line.
\end{proof}

Let now $\eps_{0}(K_{*})$, $C_{\mathrm{ml}}$, $\gamma$ be the
constants of Lemma~\ref{lem:memloss} and let $C_{2}$ be the constant from Lemma~\ref{lem:comp2}.   For the parameter
$K_{*}=C_{u}K_{1}$, set $K_{2}':=K_{2}+C_{2}K_{1}$ and recall the set
\[
\mathcal D:=\Bigl\{h\ge0:\ \textstyle\int h\,dm=1,\
\norm h_{\W1}\le K_{1}\Bigr\}\supset\bar B
\]
of Theorem~\ref{thm:main}. The next lemma is the contraction estimate on which
the main theorem rests. We stress that its two densities in the lemma belong to two different sets. The first is required to lie in $\bar B$, since the proof applies
Lemma~\ref{lem:comp2} to it and therefore needs a bound on its $\W2$-norm, while
the second is only required to lie in the larger set $\mathcal D$. 
\begin{lemma}\label{lem:stab}
Let
\[
\eps_{*}:=\min\Bigl(\eps_{0}(K_{*}),\;
\frac{1-\gamma}{4\,\hat C\,K_{2}'\,C_{\mathrm{ml}}}\Bigr),
\qquad\rho:=\frac{1+\gamma}{2}\in(0,1).
\]
Then for every $\eps\in[0,\eps_{*})$, every $h\in\bar B$ and every
$\tilde h\in\mathcal D$ the following estimate holds:
\[
\norm{\Le^{n}(h)-\Le^{n}(\tilde h)}_{\W1}
\le2C_{\mathrm{ml}}\,\rho^{\,n}\,\norm{h-\tilde h}_{\W1},
\qquad n\ge0 .
\]
\end{lemma}

\begin{proof}
Below we use the following notation
$h^{(j)}:=\Le^{j}(h)$, $\tilde h^{(j)}:=\Le^{j}(\tilde h)$ and
\begin{equation}\label{eq:traj}
P^{j}_{h}:=\Lop{h^{(j-1)}}\cdots\Lop{h^{(0)}},
\qquad\text{so that}\qquad \Le^{j}(h)=P^{j}_{h}h .
\end{equation}
and $d_{j}:=\norm{h^{(j)}-\tilde h^{(j)}}_{\W1}$. By
Proposition~\ref{prop:inv} we have $h^{(j)}\in\bar B\subset B(K_{*})$ for all
$j\ge0$, and, since $\norm{\tilde h}_{\W1}\le K_{1}$, also
$\tilde h^{(j)}\in B(K_{*})$ for all $j\ge0$. Write
\[
\Le^{n}(h)-\Le^{n}(\tilde h)
=P^{n}_{\tilde h}(h-\tilde h)+(P^{n}_{h}-P^{n}_{\tilde h})h .
\]
The function $h-\tilde h$ is mean-zero and the densities defining
$P^{n}_{\tilde h}$ lie in $B(K_{*})$, so Lemma~\ref{lem:memloss} gives $$\norm{P^{n}_{\tilde h}(h-\tilde h)}_{\W1}
\le C_{\mathrm{ml}}\gamma^{n}d_{0}.$$

For the second term we telescope:
\[
P^{n}_{h}-P^{n}_{\tilde h}
=\sum_{i=1}^{n}\Lop{h^{(n-1)}}\cdots\Lop{h^{(i)}}
\bigl(\Lop{h^{(i-1)}}-\Lop{\tilde h^{(i-1)}}\bigr)
\Lop{\tilde h^{(i-2)}}\cdots\Lop{\tilde h^{(0)}} .
\]
For $i>1$ and set
$g_{i}:=\Lop{\tilde h^{(i-2)}}\cdots\Lop{\tilde h^{(0)}}h$ and for $i=1$ we let $g_1 = h$. Here the hypothesis $h\in\bar B$ is used: since $h\in\bar B\subset\W2$ with
$\norm h_{\W2}\le K_{2}$ and $\norm h_{\W1}\le K_{1}$, Lemma~\ref{lem:comp2}
gives $\norm{g_{i}}_{\W2}\le K_{2}+C_{2}K_{1}=K_{2}'$. By (A10),
$$\norm{(\Lop{h^{(i-1)}}-\Lop{\tilde h^{(i-1)}})g_{i}}_{\W1}
\le\hat C\eps\,d_{i-1}K_{2}',$$ 
and this function is mean-zero, being a difference of two transfer operators
applied to the same function. The prefix consists of $n-i$ operators whose
densities lie in $\bar B\subset B(K_{*})$, so Lemma~\ref{lem:memloss} applies
again (including the case $n-i=0$):
\[
\norm{(P^{n}_{h}-P^{n}_{\tilde h})h}_{\W1}
\le\sum_{i=1}^{n}C_{\mathrm{ml}}\gamma^{\,n-i}\hat C\eps K_{2}'\,d_{i-1}.
\]
Hence $d_{n}\le a\gamma^{n}d_{0}+b\sum_{i=1}^{n}\gamma^{n-i}d_{i-1}$
with $a:=C_{\mathrm{ml}}\ge1$,
$b:=C_{\mathrm{ml}}\hat CK_{2}'\eps\le\frac{1-\gamma}{4}$ for $\eps<\eps_{*}$.
Lemma~\ref{lem:gronwall} now applies and finishes the proof. 
\end{proof}

\begin{remark}
Taking $n$ with $2C_{\mathrm{ml}}\rho^{n}<1$ shows that $\Le^{n}$ is a
contraction on $\bar B$ for the $\W{1}$-metric.
\end{remark}

\subsection{Proof of the main theorem}\label{ss:mainproof}
In this subsection we prove Theorem~\ref{thm:main}. 
\begin{proof}
\emph{Existence of invariant density.} Let $g_{n}:=\Le^{n}(\mathbf1)$; by
Proposition~\ref{prop:inv}, $g_{n}\in\bar B$ for all $n$. For
$n,k\in\mathbb N$ applying Lemma~\ref{lem:stab} with $h:=\Le^{k}(\mathbf1)
\in\bar B$ and $\tilde h:=\mathbf1\in\mathcal D$ yields 
\[
\norm{g_{n+k}-g_{n}}_{\W{1}}
=\norm{\Le^{n}(\Le^{k}\mathbf1)-\Le^{n}(\mathbf1)}_{\W{1}}
\le 2C_{\mathrm{ml}}\rho^{\,n}\norm{\Le^{k}\mathbf1-\mathbf1}_{\W{1}}
\le 4C_{\mathrm{ml}}K_{1}\rho^{\,n}.
\]
Hence $(g_{n})$ is Cauchy in the Banach space $\W{1}$. Therefore, it converges to
some $h_{\varepsilon}\in\W{1}$ with
$\norm{h_{\varepsilon}}_{\W{1}}\le K_{1}$. Convergence in $\W{1}$ implies
convergence in $L^{1}$, so $h_{\varepsilon}\ge0$ a.e.\ and
$\int h_{\varepsilon}\,dm=1$. Moreover, since $\norm{g_{n}}_{1,1/p}\le K'$ for every $n$ and the ball
$\{\norm{\cdot}_{1,1/p}\le K'\}$ is closed in $L^{1}$ by Lemma~\ref{lem:lsc}, we
also have $\norm{h_{\varepsilon}}_{1,1/p}\le K'$. In particular
$h_{\varepsilon}\in\mathcal D$. We claim that this limit is a fixed point of the
nonlinear operator $\Le$. Indeed, we have
$$\norm{h_{\varepsilon}-\Le(h_{\varepsilon})}_{\W{1}}
\le\norm{h_{\varepsilon}-g_{n+1}}_{\W{1}}
+\norm{\Le(g_{n})-\Le(h_{\varepsilon})}_{\W{1}}.$$ 
For the second term, decompose
\[
\Le(g_{n})-\Le(h_{\varepsilon})
=\Lop{h_{\varepsilon}}(g_{n}-h_{\varepsilon})
+\bigl(\Lop{g_{n}}-\Lop{h_{\varepsilon}}\bigr)g_{n}.
\]
By \eqref{eq:s1} and \eqref{eq:embed} we have 
$$\norm{\Lop{h_{\varepsilon}}(g_{n}-h_{\varepsilon})}_{\W{1}}
\le(\lambda+M)\norm{g_{n}-h_{\varepsilon}}_{\W{1}}.$$ 
Next we apply (A10) with
$g:=g_{n}\in\bar B$ (so $\norm{g_{n}}_{\W{2}}\le K_{2}$) and 
$$\norm{(\Lop{g_{n}}-\Lop{h_{\varepsilon}})g_{n}}_{\W{1}}
\le\hat C\varepsilon K_{2}\norm{g_{n}-h_{\varepsilon}}_{\W{1}}.$$ Both terms
tend to $0$ as $n\to\infty$, therefore,
$\Le(h_{\varepsilon})=h_{\varepsilon}$. This proves the existence part of item (i).  

Note that the $\W{2}$-weight in (A10) was placed on $g_{n}$, whose $\W{2}$-norm
is controlled by $K_{2}$. 

\emph{(ii) Exponential convergence.} Let $h\in\mathcal D$. We apply
Lemma~\ref{lem:stab} with first argument $\mathbf 1\in\bar B$ and second argument
$h\in\mathcal D$ and obtain 
\[\begin{aligned}
\norm{\Le^{n}(h)-h_{\varepsilon}}_{\W{1}}
& \le\norm{\Le^{n}(h)-g_{n}}_{\W{1}}+\norm{g_{n}-h_{\varepsilon}}_{\W{1}}\\
&\le 2C_{\mathrm{ml}}\rho^{n}\norm{h-\mathbf1}_{\W{1}}
+4C_{\mathrm{ml}}K_{1}\rho^{n}
\le C\rho^{n},
\end{aligned}\]
with $C:=2C_{\mathrm{ml}}(K_{1}+1)+4C_{\mathrm{ml}}K_{1}$, where we have used
$$\norm{g_{n}-h_{\varepsilon}}_{\W{1}}\le4C_{\mathrm{ml}}K_{1}\rho^{n}.$$

\emph{Uniqueness.} If $\tilde h\in\mathcal D$ satisfies
$\Le(\tilde h)=\tilde h$, then $\Le^{n}(\tilde h)=\tilde h$ for all
$n$, and (ii) implies that 
$$\norm{\tilde h-h_{\varepsilon}}_{\W{1}}\le C\rho^{n}\to0,$$ so
$\tilde h=h_{\varepsilon}$.
\end{proof}

%\appendix
\section{Appendix}
In this section we gather results concerning the Lasota - Yorke type estimates and some inequalities for BV type spaces. 
\subsection{Spaces of functions of generalised bounded variation, and the uniform sup-norm bound}\label{ss:keller}\label{app:A}

Let $S_{\rho}(x):=\{y\in I:\abs{x-y}<\rho\}$. Following \cite{Ke85}, for
$h:I\to\mathbb R$, we define the following
$$\operatorname{osc}(h,\rho,x):=
\operatorname*{ess\,sup}\{\abs{h(y_{1})-h(y_{2})}:y_{1},y_{2}\in
S_{\rho}(x)\},$$
and 
$$\operatorname{osc}_{1}(h,\rho):=\norm{\operatorname{osc}(h,\rho,\cdot)}_{1}.$$
Fix $\rho_{0}>0$. For $p\ge1$ let $\bv\subset L^{1}$ be the
Banach space with norm
$$\norm h_{1,1/p}=\operatorname{var}_{1,1/p}(h)+\norm h_{1}, \text{ where }
\operatorname{var}_{1,1/p}(h)
=\sup_{0<\rho\le\rho_{0}}\operatorname{osc}_{1}(h,\rho)/\rho^{1/p}.$$
It is known that the unit ball of $\bv$ is relatively compact in $L^{1}$ (see
\cite{Ke85}). Moreover the inclusion $\W1\hookrightarrow\bv$ is continuous, i.e.,
for $f\in\W1$,
$$\operatorname{osc}(f,\rho,x)\le\int_{S_{\rho}(x)}\abs{f'}\,dm,$$
hence $\operatorname{osc}_{1}(f,\rho)\le2\rho\norm{f'}_{1}$ and
$\operatorname{var}_{1,1/p}(f)\le2\rho_{0}^{1-1/p}\norm{f'}_{1}$.
Recall that $p=2$ is fixed throughout the paper; we keep the notation $1/p$ in
order to display the dependence of the exponent window on $p$.

The following elementary fact is the only closedness property used in the proof of
Theorem~\ref{thm:main}. We isolate it because relative compactness of the unit
ball, as stated in \cite{Ke85}, does not by itself imply that balls are closed.

\begin{lemma}\label{lem:lsc}
The seminorm $\operatorname{var}_{1,1/p}$ is lower semi-continuous with respect to
convergence in $L^{1}$. That is,  if $f_{n}\to f$ in $L^{1}$, then
$\operatorname{var}_{1,1/p}(f)\le\liminf_{n\to\infty}\operatorname{var}_{1,1/p}(f_{n})$.
Consequently, for every $R>0$ the ball $\{g\in L^{1}:\norm g_{1,1/p}\le R\}$ is
closed in $L^{1}$.
\end{lemma}

\begin{proof}
Fix $\rho\in(0,\rho_{0}]$. Passing  to a subsequence if necessary, we obtain a sequence $f_{n}$  along which
$\operatorname{osc}_{1}(f_{n},\rho)$ converges to its lower limit and
$f_{n}\to f$ almost everywhere. For almost every $x$ and almost every pair
$y_{1},y_{2}\in S_{\rho}(x)$,
\[
\abs{f(y_{1})-f(y_{2})}=\lim_{n\to\infty}\abs{f_{n}(y_{1})-f_{n}(y_{2})}
\le\liminf_{n\to\infty}\operatorname{osc}(f_{n},\rho,x),
\]
so that $\operatorname{osc}(f,\rho,x)\le\liminf_{n}\operatorname{osc}(f_{n},\rho,x)$
for almost every $x$. Fatou's lemma then gives
$$\operatorname{osc}_{1}(f,\rho)\le\liminf_{n}\operatorname{osc}_{1}(f_{n},\rho).$$
Dividing by $\rho^{1/p}$ shows that $f\mapsto\operatorname{osc}_{1}(f,\rho)/\rho^{1/p}$
is lower semi-continuous on $L^{1}$ for each fixed $\rho$. Since  a supremum of lower
semi-continuous functions is lower semi-continuous, taking the supremum over
$\rho\in(0,\rho_{0}]$ proves the first assertion. Since $\norm\cdot_{1}$ is
continuous on $L^{1}$, the norm $\norm\cdot_{1,1/p}$ is lower semi-continuous on
$L^{1}$, and therefore its sublevel sets are closed.
\end{proof}

\begin{proposition}\label{prop:LYBVapp}
Under the assumptions {\rm(A1)--(A8)} there are $\vartheta\in(0,1)$ and
$C_{\mathrm{ap}}>0$, uniform in $(\eps,h)$, such that
$$\norm{\Lop hf}_{1,1/p}\le\vartheta\norm f_{1,1/p}+C_{\mathrm{ap}}\norm f_{1}$$
for all $f\in\bv$.
\end{proposition}

One may take $\vartheta=\theta^{-1/2}$. We give the proof for completeness. 

We write $w_{T}:=1/\abs{T'}$ throughout the proof with the convention
$w_{T}(c_{j}):=0$, and recall that $p=2$, so that
$\var_{1,1/2}(f)=\sup_{0<\rho\le\rho_{0}}\osc_{1}(f,\rho)\rho^{-1/2}$. Further, let $g_{i}:=(T|_{I_{i}})^{-1}:J_{i}\to I_{i}$ denote the inverse branches, where
$J_{i}:=T(I_{i})$, and write
\[
\Lop hf=\sum_{i=0}^{k}(f\circ g_{i})\,w_{i}\,\mathbf 1_{J_{i}},
\qquad w_{i}:=\abs{g_{i}'}=\frac{1}{\abs{T'\circ g_{i}}}.
\]

\begin{lemma}
\label{lem:UBV}
Assume {\rm(A1)--(A8)} and set
\[
\alpha:=-\max_{1\le j\le k}\beta_{j}\ \in\ \bigl(\tfrac34,1\bigr).
\]
There is $C_{\mathrm H}<\infty$, independent of $\eps$ and $h$, such that
\[
\abs{w_{T}(x)-w_{T}(y)}\le C_{\mathrm H}\abs{x-y}^{\alpha},
\qquad x,y\in[0,1].
\]
Consequently $\var_{2}(w_{T})\le C_{\mathrm H}$ uniformly in $(\eps,h)$, and the weights $w_i$ lie in $\bv$ with uniformly bounded seminorm.
\end{lemma}

\begin{proof}
Write $\alpha_{j}:=-\beta_{j}$, so that $\alpha_{j}\ge\alpha>\tfrac34$ for every
$j$ and $\alpha_{j}<1$.

On $V$ the bounds \eqref{eq:A8far} give
$\abs{w_{T}'}=\abs{T''}/\abs{T'}^{2}\le C\theta^{-2}$, so $w_{T}$ is uniformly
Lipschitz, hence uniformly $\alpha$-H\"older, on each component of $V$.

Fix $j$ and let $x,y\in U_{j}$. By \eqref{eq:A8app},
\[
\abs{w_{T}'(x)}=\frac{\abs{T''(x)}}{\abs{T'(x)}^{2}}
\le\frac{C\abs{x-c_{j}}^{\beta_{j}-1}}{C^{-2}\abs{x-c_{j}}^{2\beta_{j}}}
=C^{3}\abs{x-c_{j}}^{-\beta_{j}-1}
=C^{3}\abs{x-c_{j}}^{\alpha_{j}-1}.
\]
If $x$ and $y$ lie on the same side of $c_{j}$, say
$0<\abs{x-c_{j}}\le\abs{y-c_{j}}$, then integrating and using the concavity of
$t\mapsto t^{\alpha_{j}}$ on $[0,\infty)$, which gives
$b^{\alpha_{j}}-a^{\alpha_{j}}\le(b-a)^{\alpha_{j}}$ for $0\le a\le b$. Since $\abs{x-y}\le1$ and $\alpha_{j}\ge\alpha$ we have 
\[
\abs{w_{T}(x)-w_{T}(y)}
\le C^{3}\int_{\abs{x-c_{j}}}^{\abs{y-c_{j}}}t^{\alpha_{j}-1}\,dt
\le\frac{C^{3}}{\alpha_{j}}\abs{x-y}^{\alpha_{j}}
\le\frac{C^{3}}{\alpha}\abs{x-y}^{\alpha}.
\]
 Letting
$y\to c_{j}$ and using $w_{T}(c_{j})=0$, which is consistent with (A6), the same
bound gives $\abs{w_{T}(x)}\le\frac{C^{3}}{\alpha}\abs{x-c_{j}}^{\alpha}$.  Hence
if $x$ and $y$ lie on opposite sides of $c_{j}$,
\[
\abs{w_{T}(x)-w_{T}(y)}\le\abs{w_{T}(x)}+\abs{w_{T}(y)}
\le\tfrac{C^{3}}{\alpha}\bigl(\abs{x-c_{j}}^{\alpha}+\abs{y-c_{j}}^{\alpha}\bigr)
\le\tfrac{2C^{3}}{\alpha}\abs{x-y}^{\alpha}.
\]
For general $x<y$ in $[0,1]$, split $[x,y]$ at the endpoints of the $U_{j}$ it
meets. This produces at most $2k+1$ subintervals, on each of which the above
applies, so by the power-mean inequality
\[
\abs{w_{T}(x)-w_{T}(y)}\le C'\sum_{i}\abs{x_{i}-x_{i-1}}^{\alpha}
\le C'(2k+1)^{1-\alpha}\abs{x-y}^{\alpha}.
\]
This proves the H\"older bound with a constant $C_{\mathrm H}$ depending only on
$C$, $\theta$, $\alpha$ and $k$.

For the $2$-variation, let $0=x_{0}<\cdots<x_{N}=1$ be any partition. Then
\[
\sum_{i=1}^{N}\abs{w_{T}(x_{i})-w_{T}(x_{i-1})}^{2}
\le C_{\mathrm H}^{2}\sum_{i=1}^{N}\abs{x_{i}-x_{i-1}}^{2\alpha}
\le C_{\mathrm H}^{2}\Bigl(\sum_{i=1}^{N}\abs{x_{i}-x_{i-1}}\Bigr)^{2\alpha}
=C_{\mathrm H}^{2},
\]
the middle step by superadditivity of $t\mapsto t^{2\alpha}$, valid since
$2\alpha>\tfrac32>1$. Taking the supremum over partitions gives
$\var_{2}(w_{T})\le C_{\mathrm H}$. The last statement of the Lemma follows from \cite[Lemma 2.7]{Ke85}
\end{proof}

\begin{proof}[Proof of Proposition~\ref{prop:LYBVapp}]
By Lemma~\ref{lem:UBV} the weights $w_{i}$ in $\Lop hf$ have uniformly bounded
$\mathrm{UBV}_{2}$ norm. Every endpoint of $J_{i}$ lying in the interior of
$[0,1]$ is the image of a singular point $c_{j}$, by (A2), and at such a point
$w_{i}\to0$ by (A6); hence the extension of $w_{i}\mathbf 1_{J_{i}}$ by zero has
no interior jump, and the endpoint contributions at $0$ and $1$ are absorbed in
the lower-order term by the trace estimates \cite[Lemmas 2.1--2.2]{Ke85}.

Fix $\rho\in(0,\rho_{0}]$. Since $\abs{g_{i}'}\le\theta^{-1}$,
\cite[Lemma 2.3]{Ke85} gives
$\osc(f\circ g_{i},\rho,z)\le\osc(f,\rho/\theta,g_{i}(z))$ for a.e.\ $z\in J_{i}$,
and \cite[Lemma 2.4]{Ke85} gives
\[
\osc_{1}\bigl((f\circ g_{i})w_{i}\mathbf 1_{J_{i}},\rho\bigr)
\le\int_{J_{i}}\osc\bigl(f\circ g_{i},\rho,z\bigr)w_{i}(z)\,dz
+C_{i}\rho^{1/2}\norm f_{1}.
\]
Changing variables $y=g_{i}(z)$, so that $dy=w_{i}(z)\,dz$, the integral is at
most $\int_{I_{i}}\osc(f,\rho/\theta,y)\,dy$. Summing over the branches and using
that the $I_{i}$ partition $[0,1]$,
\begin{equation}\label{eq:nobranchcount}
\sum_{i=0}^{k}\int_{J_{i}}\osc\bigl(f\circ g_{i},\rho,z\bigr)w_{i}(z)\,dz
\le\int_{0}^{1}\osc(f,\rho/\theta,y)\,dy=\osc_{1}(f,\rho/\theta).
\end{equation}
Notice that after the change of variables the leading terms
reassemble into a single integral over $[0,1]$, so the number $k+1$ of branches
does \emph{not} multiply the leading constant, and enters only through
$C_{\mathrm{ap}}:=\sum_{i}C_{i}$. Since $\theta>1$ we have
$\rho/\theta<\rho\le\rho_{0}$, so
$\osc_{1}(f,\rho/\theta)\le(\rho/\theta)^{1/2}\var_{1,1/2}(f)$ and therefore
\[
\frac{\osc_{1}(\Lop hf,\rho)}{\rho^{1/2}}
\le\theta^{-1/2}\var_{1,1/2}(f)+C_{\mathrm{ap}}\norm f_{1}.
\]
Taking the supremum over $\rho\in(0,\rho_{0}]$ and adding
$\norm{\Lop hf}_{1}\le\norm f_{1}$ gives the assertion with
$\vartheta=\theta^{-1/2}$, after enlarging $C_{\mathrm{ap}}$ by $1$.
\end{proof}

\begin{proposition}\label{prop:supbd}
For any $f\in\bv$, $n\in\mathbb N$ and admissible
$h_{1},\dots,h_{n}$,
$$\norm{\Lop{h_{n}}\cdots\Lop{h_{1}}f}_{\infty}\le\hat M\norm f_{\infty}$$
with $\hat M$ uniform.
\end{proposition}

\begin{proof}
By Proposition~\ref{prop:LYBVapp} and $\norm{\Lop hf}_{1}\le\norm f_{1}$,
iterating gives
$\norm{\Lop{h_{n}}\cdots\Lop{h_{1}}f}_{1,1/p}
\le\vartheta^{n}\norm f_{1,1/p}+\tfrac{C_{\mathrm{ap}}}{1-\vartheta}\norm f_{1}$.
For $f\equiv\mathbf1$ this is bounded by a uniform $M_{0}'$; since
$\norm\cdot_{\infty}\le\hat C_{\infty}\norm\cdot_{1,1/p}$,
$\norm{\Lop{h_{n}}\cdots\Lop{h_{1}}\mathbf1}_{\infty}\le
\hat C_{\infty}M_{0}'=:\hat M$. For general $f$, positivity gives
$\abs{\Lop{h_{n}}\cdots\Lop{h_{1}}f}\le
\norm f_{\infty}\Lop{h_{n}}\cdots\Lop{h_{1}}\mathbf1$, whence the
claim.
\end{proof}

\subsection{Proof of Proposition~\ref{prop:LY}}\label{app:B}
Here we prove the main technical result of this work: Proposition~\ref{prop:LY}. 

First of all, notice that the local part (A8)(i) provides a constant $C\ge1$ such
that, for every $j$ and every $x\in U_{j}\setminus\{c_{j}\}$,
\begin{equation}\label{eq:A8app}
\abs{T''(x)}\le C\abs{x-c_{j}}^{\beta_{j}-1},\quad
\abs{T'''(x)}\le C\abs{x-c_{j}}^{\beta_{j}-2},\quad
\abs{T'(x)}\ge C^{-1}\abs{x-c_{j}}^{\beta_{j}} ,
\end{equation}
while (A8)(ii) together with (A4) provides the uniform bounds
\begin{equation}\label{eq:A8far}
\abs{T''(x)}\le C,\qquad
\abs{T'''(x)}\le C,\qquad
\abs{T'(x)}\ge\theta>1,
\qquad x\in V .
\end{equation}
Recall from \eqref{eq:r0} that $U_{1},\dots,U_{k}$ are pairwise disjoint and that
$V=[0,1]\setminus\bigcup_{j}U_{j}$, so that \eqref{eq:A8app} and
\eqref{eq:A8far} together cover all of $[0,1]\setminus\{c_{1},\dots,c_{k}\}$.
Only $\beta_{j}<-\tfrac12$ is used below.

Fix $(\eps,h)$ and write $T:=T_{\eps,h}$, $\LT:=\Lop h$. By
(A1), (A3) each $T_{i}:=T|_{I_{i}}$ is strictly monotone; set
$J_{i}:=T_{i}(I_{i})$, $g_{i}:=T_{i}^{-1}\in C^{3}(J_{i})$ with sign
$s_{i}=\operatorname{sgn}g_{i}'$, and
\begin{equation}\label{eq:ginv}
g_{i}'=\frac{1}{T'\circ g_{i}},\qquad
g_{i}''=-\frac{T''\circ g_{i}}{(T'\circ g_{i})^{3}},\qquad
g_{i}'''=-\frac{T'''\circ g_{i}}{(T'\circ g_{i})^{4}}
+\frac{3(T''\circ g_{i})^{2}}{(T'\circ g_{i})^{5}},
\end{equation}
with $\abs{g_{i}'}\le\theta^{-1}$. The transfer operator can be written as 
\begin{equation}\label{eq:explicit}
\LT f=\sum_{i=0}^{k}\phi_{i},\qquad
\phi_{i}:=\bigl((f\circ g_{i})\abs{g_{i}'}\bigr)\mathbf1_{J_{i}}
\ \text{(extended by $0$)}.
\end{equation}
Moreover $\LT$ is positive,
$\norm{\LT f}_{1}\le\norm f_{1}$ and $\abs{\LT f}\le\LT\abs f$. By
(A2), every endpoint of every $J_{i}$ is either a
\emph{singular value} $a_{j}$ (the image of a singularity
$c_{j}\in\partial I_{i}$) or the boundary point
$0=T(0^{+})=T(1^{-})\in\partial[0,1]$.

\subsection{Distortion estimates}

Recall $r_{0}$, $U_{j}$ and $V$ from \eqref{eq:r0}, and set
\[
W:=\frac{\abs{T''}}{\abs{T'}^{2}},\qquad
E_{1}:=\frac{\abs{T''}}{\abs{T'}^{3}},\qquad
E_{2}:=\frac{\abs{T'''}}{\abs{T'}^{3}}
+\frac{\abs{T''}^{2}}{\abs{T'}^{4}} .
\]

\begin{sublemma}\label{sub:weights}
There are uniform $M_{1},M_{2},M_{3}<\infty$ with:
\begin{itemize}
    \item[(a)] $\norm W_{L^{2}(0,1)}\le M_{1}$;
    \item[(b)] $\norm{E_{1}}_{\infty}\le M_{2}$ and
$E_{1}(x)\to0$ as $x\to c_{j}^{\pm}$, every $j$;
\item[(c)] $\norm{E_{2}}_{L^{1}(0,1)}\le M_{3}$;
\item[(d)] $1/\abs{T'(x)}\to0$ as $x\to c_{j}^{\pm}$, every $j$.
\item[(e)]Moreover $\operatorname{osc}_{1}(1/\abs{T'},\rho)\le C\rho$ uniformly
for $0<\rho\le\rho_{0}$.
\end{itemize}
\end{sublemma}

\begin{proof}
On $V$ the bounds \eqref{eq:A8far} give $\abs{T''}\le C$, $\abs{T'''}\le C$ and
$\abs{T'}\ge\theta>1$, whence
\[
W\le C\theta^{-2},\qquad E_{1}\le C\theta^{-3},\qquad
E_{2}\le C\theta^{-3}+C^{2}\theta^{-4}\qquad\text{on }V,
\]
so that $W,E_{1},E_{2}$ are uniformly bounded there. This is where the local
formulation of (A8) is used.

On $U_{j}$, with $t:=\abs{x-c_{j}}\in(0,r_{0})$, \eqref{eq:A8app}
gives
\begin{equation}\label{eq:peaks}
W\le C^{3}t^{-(1+\beta_{j})},\qquad
E_{1}\le C^{4}t^{-(1+2\beta_{j})},\qquad
E_{2}\le2C^{6}t^{-(2+2\beta_{j})}.
\end{equation}
Since $\beta_{j}<-\tfrac12$ we have $2(1+\beta_{j})<1$, so
$W^{2}$ is integrable around $c_{j}$, giving (a). Further, since 
$-(1+2\beta_{j})>0$,  $E_{1}$ is bounded on $U_{j}$ and tends to
$0$ at $c_{j}$, giving (b). Also, we have $2+2\beta_{j}<1$, which implies  $E_{2}$ is
integrable around $c_{j}$, giving (c). Item (d) follows from (A6). To  the
oscillation bound, we write $u:=1/\abs{T'}$. Then $\abs{u'}=W$ off the
singular set, so for $x$ with $d(x):=\min_{j}\abs{x-c_{j}}\ge2\rho$ letting
$\beta_{*}:=\max_{j}\beta_{j}<-\tfrac12$ we have 
$$\operatorname{osc}(u,\rho,x)\le2\rho\sup_{S_{\rho}(x)}W\le
C\rho\,d(x)^{-(1+\beta_{*})}.$$ Further, for $x$ with $d(x)\le2\rho$ holds 
$$\operatorname{osc}(u,\rho,x)\le2\sup_{S_{\rho}(x)}u\le
C\rho^{-\beta_{*}}$$ by \eqref{eq:A8app}. Integrating in $x$, using $1+\beta_{*}<1$  and
$1-\beta_{*}>1$ implies
\[
\operatorname{osc}_{1}(u,\rho)
\le Ck\rho\cdot\rho^{-\beta_{*}}
+C\rho\int_{\{d\ge2\rho\}}d(x)^{-(1+\beta_{*})}dx
\le C'\rho^{1-\beta_{*}}+C'\rho\le C''\rho.
\]
\end{proof}

\subsection{Proof of \eqref{eq:L2}}

For $g\in L^{\infty}$, letting $f=\mathbf1$ we obtain 
$\int(g\circ T)^{2}dm=\int\LT\mathbf1\cdot g^{2}dm$. Using the representation \eqref{eq:explicit} we obtain 
$\LT\mathbf1(y)=\sum_{x\in T^{-1}\{y\}}1/\abs{T'(x)}\le(k+1)/\theta$
($k+1$ branches, by (A4)). Hence
$\norm{g\circ T}_{2}\le((k+1)/\theta)^{1/2}\norm g_{2}$. For
positive $f\in L^{2}$ duality gives
\[
\norm{\LT f}_{2}
=\sup_{\substack{g\in L^{\infty},\,g\ge0\\ \norm g_{2}\le1}}
\int\LT f\cdot g\,dm
=\sup_{\substack{g\in L^{\infty},\,g\ge0\\ \norm g_{2}\le1}}\int f\cdot g\circ T\,dm
\le\Bigl(\tfrac{k+1}{\theta}\Bigr)^{1/2}\norm f_{2}.
\]
For arbitrary  $f\in L^{2}$  we use $\abs{\LT f}\le\LT\abs f$ and the above inequality. \qed

\subsection{Proof of \eqref{eq:LY1}}

Let $f\in\W1$.  On a compact subset $J_{i}$ of $V$ the composition $f\circ g_{i}$
is absolutely continuous (AC),
$\abs{g_{i}'}=s_{i}g_{i}'\in C^{1}$, so $\phi_{i}$ is locally AC with
\begin{equation}\label{eq:phiprime}
\phi_{i}'=s_{i}\bigl[(f'\circ g_{i})(g_{i}')^{2}
+(f\circ g_{i})g_{i}''\bigr]\quad\text{a.e.\ on }J_{i}.
\end{equation}

Changing variables $y=T_{i}(x)$ and using
\eqref{eq:ginv},
\begin{equation}\label{eq:t12}
\int_{J_{i}}\abs{f'\circ g_{i}}(g_{i}')^{2}dy
=\int_{I_{i}}\frac{\abs{f'}}{\abs{T'}}dm
\le\theta^{-1}\int_{I_{i}}\abs{f'}dm,
\qquad
\int_{J_{i}}\abs{f\circ g_{i}}\abs{g_{i}''}dy
=\int_{I_{i}}\abs f\,W\,dm .
\end{equation}

\emph{Boundary behaviour.} $\phi_{i}$ is locally AC on $J_{i}$ with
$L^{1}$ derivative, hence AC on $\overline{J_{i}}$ with one-sided
limits at the endpoints. At a singular endpoint $a_{j}$,
$g_{i}(y)\to c_{j}$ and
$\abs{\phi_{i}(y)}\le\norm f_{\infty}/\abs{T'(g_{i}(y))}\to0$ by
Sublemma~\ref{sub:weights}(d); every non-singular endpoint equals
$0\in\partial[0,1]$. Thus the zero-extension of $\phi_{i}$ is
continuous at every endpoint of $J_{i}$ interior to $(0,1)$ and has
integrable a.e.-derivative supported in $J_{i}$: it lies in
$\W1(0,1)$, with distributional derivative \eqref{eq:phiprime} and no
atoms.

 Summing \eqref{eq:t12} over $i$ and using
Cauchy--Schwarz with Sublemma~\ref{sub:weights}(a),
\[
\norm{(\LT f)'}_{1}
\le\theta^{-1}\norm{f'}_{1}+\int_{0}^{1}\abs f\,W\,dm
\le\theta^{-1}\norm{f'}_{1}+M_{1}\norm f_{2}.
\]
This is \eqref{eq:LY1} with $\lambda:=\theta^{-1}$, $M_{0}\ge M_{1}$.
\qed

\subsection{Proof of \eqref{eq:LY2}}

For $f\in\W2$ recall that 
$\norm f_{\infty},\norm{f'}_{\infty}\le\norm f_{\W2}$.

On compacts of $J_{i}$ both
summands of \eqref{eq:phiprime} are AC, and
\begin{equation}\label{eq:phisecond}
\phi_{i}''=s_{i}\bigl[(f''\circ g_{i})(g_{i}')^{3}
+3(f'\circ g_{i})g_{i}'g_{i}''
+(f\circ g_{i})g_{i}'''\bigr]\quad\text{a.e.\ on }J_{i}.
\end{equation}

 Changing variables and using \eqref{eq:ginv}:
\[
\int_{J_{i}}\abs{f''\circ g_{i}}\abs{g_{i}'}^{3}dy
=\int_{I_{i}}\frac{\abs{f''}}{\abs{T'}^{2}}dm
\le\theta^{-2}\int_{I_{i}}\abs{f''}dm;
\]
since $\abs{g_{i}'g_{i}''}\circ T\cdot\abs{T'}
=\abs{T''}/\abs{T'}^{3}=E_{1}$,
\[
3\int_{J_{i}}\abs{f'\circ g_{i}}\abs{g_{i}'}\abs{g_{i}''}dy
=3\int_{I_{i}}\abs{f'}E_{1}\,dm\le3M_{2}\int_{I_{i}}\abs{f'}dm;
\]
and since $\abs{g_{i}'''}\circ T\cdot\abs{T'}
\le\abs{T'''}/\abs{T'}^{3}+3\abs{T''}^{2}/\abs{T'}^{4}\le3E_{2}$,
\[
\int_{J_{i}}\abs{f\circ g_{i}}\abs{g_{i}'''}dy
\le3\norm f_{\infty}\int_{I_{i}}E_{2}\,dm .
\]

Finally we show that  $\phi_i\in \W2$. Since $\phi_{i}'$ is locally AC on $J_{i}$
with $L^{1}$ derivative, it has one-sided limits at the endpoints
of $J_{i}$. At a singular endpoint $a_{j}$ both terms of
\eqref{eq:phiprime} vanish in the limit:
$$\abs{(f'\circ g_{i})(g_{i}')^{2}}
\le\norm{f'}_{\infty}/\abs{T'\circ g_{i}}^{2}\to0$$ by
Sublemma~\ref{sub:weights}(d), and
$$\abs{(f\circ g_{i})g_{i}''}
\le\norm f_{\infty}\,E_{1}(g_{i}(y))\to0$$ by
Sublemma~\ref{sub:weights}(b). Since $\beta_{j}<-\tfrac12$ the exponent
$-(1+2\beta_{j})$ of $E_{1}$ near $c_{j}$ is positive. Every non-singular endpoint lies on
$\partial[0,1]$. As in the proof of \eqref{eq:LY1}, applied to $\phi_{i}'$,
the zero-extension of $\phi_{i}'$ lies in $\W1(0,1)$ with
distributional derivative \eqref{eq:phisecond} and no atoms; hence
the zero-extension of $\phi_{i}$ lies in $\W2(0,1)$.

Summing over $i$ we obtain
\[
\norm{(\LT f)''}_{1}
\le\theta^{-2}\norm{f''}_{1}+3M_{2}\norm{f'}_{1}
+3M_{3}\norm f_{\infty}
\le\lambda^{2}\norm{f''}_{1}+(3M_{2}+3M_{3})\norm f_{\W1},
\]
Combining the above estimate with  $\norm{\LT f}_{1}\le\norm f_{1}$
and the estimate \eqref{eq:LY1}  we obtain
\[
\begin{aligned}
\norm{\LT f}_{\W2}
&=\norm{\LT f}_{1}+\norm{(\LT f)'}_{1}+\norm{(\LT f)''}_{1}\\[2pt]
&\le\norm f_{1}
+\bigl(\lambda\norm{f'}_{1}+M_{1}\norm f_{2}\bigr)
+\bigl(\lambda^{2}\norm{f''}_{1}+3M_{2}\norm{f'}_{1}+3M_{3}\norm f_{\infty}\bigr)\\[2pt]
&\le\lambda^{2}\norm f_{\W2}
+\bigl(1+\lambda+M_{1}+3M_{2}+3M_{3}\bigr)\norm f_{\W1},
\end{aligned}
\]
%where we used $\norm f_{2}\le\norm f_{\infty}\le\norm f_{\W1}$ and
%$\norm f_{1},\norm{f'}_{1}\le\norm f_{\W1}$ together with
%$\norm{f''}_{1}\le\norm f_{\W2}$. 
Letting 
\begin{equation}\label{eq:M0}
M_{0}:=1+\lambda+M_{1}+3M_{2}+3M_{3},
\end{equation}
finishes the proof of   \eqref{eq:LY2}. Notice that the constant $M_{0}$ is uniform in $(\eps,h)$ and satisfies $M_{0}\ge M_{1}$, so that the same
constant may be used in \eqref{eq:LY1}. \qed

\section{Acknowledgments}
The authors thank Wael Bahsoun for suggesting this problem and his continuous support during the course of the work. Also, we thank Stefano Galatolo for suggesting a generalization of the family  of base maps  to have several singular points. We also thank anonymous referees for pointing out several mistakes in the previous version.  

\section{Funding information}
MR was supported by Research Project No. FL-9524115148 of the Ministry of Higher Education, Science and Innovation of the Republic of Uzbekistan. 

\section{Data Availability:} No datasets were generated or analysed during the current study.

\medskip\section{Declarations}

\medskip\textbf{Competing Interests:} The authors declare no competing interests.

\medskip\textbf{Ethics Declaration:} Not applicable.

\end{document}